\documentclass{proc-l}
\usepackage{verbatim,amscd,amssymb}
\usepackage{verbatim,amssymb,graphicx}
\usepackage[active]{srcltx}

\newtheorem{thm}{Theorem}[section]

\newtheorem{lem}[thm]{Lemma}

\newtheorem{cor}[thm]{Corollary}

\newtheorem*{mthm}{Main Theorem}

\theoremstyle{definition}
\newtheorem{dfn}[thm]{Definition}

\def\Ac{\mathcal{A}}  \def\Cc{\mathcal{C}}
\def\C{\mathbb{C}}  \def\Fc{\mathcal{F}}
 \def\al{\alpha}
  \def\Mc{\mathcal{M}}
\def\Pc{\mathcal{P}}  \def\Oc{\mathcal{O}}

\def\R{\mathbb{R}}

\def\Uc{\mathcal{U}} \def\Vc{\mathcal{V}} \def\Wc{\mathcal{W}}

\def\Z{\mathbb{Z}}

\renewcommand\emptyset{\varnothing}
\newcommand{\sm}{\setminus}

\def\eps{\varepsilon}
\def\La{\Lambda}

\def\al{\alpha}
\def\be{\beta}

\def\ga{\gamma}

\def\om{\omega}
\def\la{\lambda}
\def\si{\sigma}

\def\ol{\overline}

\def\thu{\mathrm{TH}}
\def\Re{\mathrm{Re}}

\def\cu{\mathrm{CU}}
\def\uc{\mathbb{S}}
\def\bd{\mathrm{Bd}}

\def\le{\leqslant}
\def\ge{\geqslant}
\def\0{\emptyset}
\def\disk{\mathbb{D}}

\def\phd{\mathrm{PHD}}

\begin{document}
\date{November 11, 2014; revised November 30, 2015; second revision January 20, 2016}
\title[The Principal Hyperbolic Domain in the cubic case]
{Complementary components to the cubic Principal Hyperbolic Domain}

\author[A.~Blokh]{Alexander~Blokh}

\author[L.~Oversteegen]{Lex Oversteegen}

\author[R.~Ptacek]{Ross~Ptacek}

\author[V.~Timorin]{Vladlen~Timorin}

\address[Alexander~Blokh, Lex~Oversteegen and Ross~Ptacek]
{Department of Mathematics\\ University of Alabama at Birmingham\\
Birmingham, AL 35294}

\address[Vladlen~Timorin]
{Faculty of Mathematics\\
Laboratory of Algebraic Geometry and its Applications\\
National Research University Higher School of Economics\\
Vavilova St. 7, 112312 Moscow, Russia }

\address[Vladlen~Timorin]
{Independent University of Moscow\\
Bolshoy Vlasyevskiy Pereulok 11, 119002 Moscow, Russia}

\email[Alexander~Blokh]{ablokh@math.uab.edu}
\email[Lex~Oversteegen]{overstee@math.uab.edu}
\email[Ross~Ptacek]{rptacek@uab.edu}
\email[Vladlen~Timorin]{vtimorin@hse.ru}

\subjclass[2010]{Primary 37F45; Secondary 37F10, 37F20, 37F50}

\keywords{Complex dynamics; Julia set; polynomial-like maps;
laminations}

%\definecolor{Red}{rgb}{1,0,0}

\begin{abstract}
We study the closure of the cubic Principal Hyperbolic Domain and
its intersection $\Pc_\la$ with the slice $\Fc_\la$ of the space of
all cubic polynomials with fixed point $0$ defined by the multiplier
$\la$ at $0$. We show that any bounded domain $\Wc$ of
$\Fc_\la\sm\Pc_\la$ consists of $J$-stable polynomials $f$ with
connected Julia sets $J(f)$ and is either of \emph{Siegel capture}
type (then $f\in \Wc$ has an invariant Siegel domain $U$ around $0$
and another Fatou domain $V$ such that $f|_V$ is two-to-one and
$f^k(V)=U$ for some $k>0$) or of \emph{queer} type (then a specially chosen
critical point of $f\in \Wc$ belongs to $J(f)$, the set $J(f)$
has positive Lebesgue measure, and carries an invariant line field).
\end{abstract}

\maketitle

\section{Introduction}
In this paper, we study topological dynamics of complex cubic
polynomials. We denote the \emph{Julia set} of a polynomial $f$ by
$J(f)$ and the \emph{filled Julia set} of $f$ by $K(f)$. Let us
recall classical facts about quadratic polynomials.
The Mandelbrot set $\Mc_2$, perhaps the most well-known mathematical
set outside of the mathematical community, can be defined as the set
of all complex numbers $c$ such that the sequence
$$
c,\quad c^2+c,\quad (c^2+c)^2+c,\dots
$$
is bounded. The numbers $c$ label polynomials $z^2+c$. Every
quadratic polynomial can be reduced to this form by an affine
coordinate change.

By definition, $c\in \Mc_2$ if the orbit of $0$ under $z\mapsto
z^2+c$ is bounded, i.e., $0\in K(z^2+c)$. Note that $0$ is the only
critical point of the polynomial $z^2+c$ in $\C$. Generally, the
behavior of critical orbits to a large extent determines the
dynamics of other orbits. For example, by a classical theorem of
Fatou and Julia, $c\in \Mc_2$ if and only if $K(z^2+c)$ is
connected. If $c\not\in \Mc_2$, then the set $K(z^2+c)$ is a Cantor
set.

The central part of the Mandelbrot set, the so called \emph{Principal
Hyperbolic Domain} $\phd_2$, is bounded by a cardioid called the
\emph{Main Cardioid}. By definition, the Principal
Hyperbolic Domain $\phd_2$ consists of all parameter values $c$ such
that the polynomial $z^2+c$ is hyperbolic, and the set $K(z^2+c)$ is
a Jordan disk (a polynomial of any degree is said to be {\em
hyperbolic} if the orbits of all its critical points converge to
attracting cycles). Equivalently, $c\in\phd_2$ if and only if
$z^2+c$ has an attracting fixed point. The closure of $\phd_2$
consists of all parameter values $c$ such that $z^2+c$ has a
non-repelling fixed point. As follows from the
Douady--Hubbard parameter landing theorem \cite{DH, Hu}, the
Mandelbrot set itself can be thought of as the union of the main
cardioid and \emph{limbs} (connected components of
$\Mc_2\sm\ol{\phd}_2$) parameterized by reduced rational fractions
$p/q\in (0,1)$.

This motivates our study of higher degree analogs of $\phd_2$ started
in \cite{bopt14}. More precisely, complex numbers $c$ are in one-to-one
correspondence with affine conjugacy classes of quadratic polynomials
(throughout we call affine conjugacy classes of polynomials
\emph{classes} of polynomials). Thus a natural higher-degree analog of
the set $\Mc_2$
is the %{\em degree $d$ Mandelbrot set}
{\em degree $d$ connectedness locus} $\Mc_d$ defined as the set of
classes of degree $d$ polynomials $f$, all of whose critical points do
not escape, or, equivalently, whose Julia set $J(f)$ is connected. The
\emph{Principal Hyperbolic Domain} $\phd_d$ of $\Mc_d$ is defined as
the set of classes of hyperbolic degree $d$ polynomials with Jordan
curve Julia sets. Equivalently, the class $[f]$ of a degree $d$
polynomial $f$ belongs to $\phd_d$ if all critical points of $f$ are in
the immediate attracting basin of the same attracting (or
super-attracting) fixed point. In \cite{bopt14} we describe properties
of cubic polynomials $f$ such that $[f]\in \ol\phd_d$; notice that
Theorem~\ref{t:prophd} holds for any $d\ge 2$.

\begin{thm}[\cite{bopt14}]\label{t:prophd}
If $[f]\in \ol{\phd}_d$, then $f$ has a non-repelling fixed point, no
repelling periodic cutpoints in $J(f)$, and all its non-repelling
periodic points, except at most one fixed point, have multiplier 1.
\end{thm}

Observe that, strictly speaking, in \cite{bopt14} we claim that all
non-repelling periodic \emph{cutpoints} in the Julia set $J(f)$, except
perhaps one, have multiplier 1; however, literally repeating the same
arguments one can prove the version of the results of \cite{bopt14}
given by Theorem~\ref{t:prophd} (i.e., we can talk about all
non-repelling periodic points of $f$, not only its periodic cutpoints).
Theorem~\ref{t:prophd} motivates the following definition; notice that
from now on in the paper we concentrate upon the cubic case (thus,
unlike Theorem~\ref{t:prophd}, Definition~\ref{d:cubio} deals with
cubic polynomials).

\begin{dfn}[\cite{bopt14}]\label{d:cubio}
Let $\cu$ be the family of classes of cubic polynomials $f$ with
connected $J(f)$ such that $f$ has a non-repelling fixed
point, no repelling periodic cutpoints in $J(f)$,  and all its
non-repelling periodic points, except at most one fixed point, have
multiplier 1. The family $\cu$ is called the \emph{Main Cubioid}.
\end{dfn}

%, and polynomial $f$ with $[f]\in \cu$ is said to be
%\emph{cubioidal}.
%\end{dfn}

Let $\Fc$ be the space of polynomials
$$
f_{\lambda,b}(z)=\lambda z+b z^2+z^3,\quad \lambda\in \C,\quad b\in \C.
$$
An affine change of variables reduces any cubic polynomial $f$ to
the form $f_{\lambda,b}$. Note that $0$ is a fixed point for every
polynomial in $\Fc$. The set of all polynomials $f\in\Fc$ such that
$0$ is non-repelling for $f$ is denoted by $\Fc_{nr}$ (in other
words, $\Fc_{nr}$ is the set of all polynomials $f_{\lambda,b}$ with
$|\la|\le 1$). Define the \emph{$\la$-slice} $\Fc_\lambda$ of $\Fc$
as the space of all polynomials $g\in\Fc$ with $g'(0)=\lambda$. The
space $\Fc$ maps onto the space of classes of all cubic polynomials
with a fixed point of multiplier $\lambda$ as a finite branched
covering. This branched covering is equivalent to the map $b\mapsto
a=b^2$, i.e., classes of polynomials $f_{\la,b}\in\Fc_\la$ are in
one-to-one correspondence with the values of $a$. Thus, if we talk
about, say, points $[f]$ of $\Mc_3$, then it suffices to take
$f\in\Fc_\lambda$ for some $\lambda$. There is no loss of generality
in that we consider only perturbations of $f$ in $\Fc$.

Assume that $J(f)$ is connected. In \cite{lyu83, MSS}, the notion of
\emph{$J$-stability} was introduced for any holomorphic family of
rational functions: a map is \emph{$J$-stable with respect to a
family of maps} if its Julia set admits an equivariant holomorphic
motion over some neighborhood of the map in the given family. We say
that $f\in \Fc_\la$ is \emph{stable} if it is $J$-stable with
respect to $\Fc_\lambda$ with $\lambda=f'(0)$, otherwise we say that
$f$ is \emph{unstable}. The set $\Fc^{st}_\la$ of all stable
polynomials $f\in \Fc_\la$ is an open subset of $\Fc_\la$. A
component of $\Fc^{st}_\la$ is called a \emph{$(\lambda$-$)$stable
component}.

\begin{dfn}\label{d:exte}
The \emph{extended closure $\ol{\phd}_3^e$ of the cubic Principal
Hyperbolic Domain $\ol{\phd}_3$} is the union of $\ol\phd_3$ and
classes of all polynomials from all $\la$-stable components $\La$
with $|\la|\le 1$ such that for all $b\in\bd(\La)$, we have
$[f_{\la,b}]\in\ol\phd_3$.
\end{dfn}

It turns out that properties of polynomials from $\ol\phd_3$ listed
in Theorem~\ref{t:prophd} are inherited by polynomials from the
extended closure $\ol{\phd}_3^e$.

\begin{thm}[\cite{bopt14}]\label{t:extendclo}
We have $\ol\phd_3^e\subset \cu$.
\end{thm}

In \cite{bopt14b} we studied polynomials $f\in \Fc_{nr}$ that have
connected quadratic-like Julia sets containing $0$. Lemma
\ref{l:2crpts}, Theorem \ref{t:princ} and Corollary \ref{c:bdd-stab}
are proven in \cite{bopt14b}; we include them in this paper for the sake of
completeness.

Our aim is to continue to study properties of $\ol\phd_3$. For a
compact set $X\subset \C$, define the \emph{topological hull
$\thu(X)$} of $X$ as the union of $X$ with all bounded components of
$\C\setminus X$. %For any set $\Wc$ of polynomials, let $[\Wc]$ be
%the set of all classes $[f]$, where $f\in\Wc$.
We will write $\Pc_\la$ for the set of polynomials $f\in\Fc_\la$
such that $[f]\in\ol\phd_3$. In our Main Theorem we
describe the dynamics of polynomials $f$ belonging to bounded
components of the set $\thu(\Pc_\la)\sm \Pc_\la$ where $|\la|\le 1$.

Consider a cubic polynomial $f$ with a non-repelling fixed point
such that $[f]$ does not belong to $\ol\phd_3$; we call such
polynomials \emph{potentially renormalizable}. Whenever we talk
about a potentially renormalizable polynomial, we always assume that
it has a non-repelling fixed point. In that case we may assume that
$f\in \Fc_{nr}$ (i.e., we have $f(z)=f_{\lambda,b}(z)=\lambda z+b z^2+z^3$ with
$|\la|\le 1$).

Let $\Ac=\bigcup_{|\la|<1} \Fc_\la$.
For each $g\in \Ac$, let $A(g)$ be the basin of $0$.
Perturbing a potentially renormalizable polynomial $f\in \Fc_{nr}$ to a polynomial $g\in \Ac$,
we see that $g|_{A(g)}$ is two-to-one (otherwise $[f]\in \ol\phd_3$). We use
this in Lemma~\ref{l:2crpts} to show that a potentially renormalizable polynomial $f\in \Fc_{nr}$
has two distinct critical points. A critical
point $c$ of $f$ is said to be \emph{principal} if there is a
neighborhood $\Uc$ of $f$ in $\Fc$ and a holomorphic function
$\omega_1:\Uc\to\C$ defined on $\Uc$ such that $c=\omega_1(f)$, and,
for every $g\in\Uc\cap\Ac$, the point $\omega_1(g)$ is the critical
point of $g$ contained in $A(g)$. By Theorem~\ref{t:princ}, if $f\in \Fc_{nr}$
is potentially renormalizable, then the
point $\om_1(f)$ is well-defined; let the other critical point of
$f$ be $\om_2(f)$. It is easy to see that $\om_1(f)\in K(f)$.

\begin{dfn}\label{d:compotypes}
Let $\Wc$ be a bounded component of $\thu(\Pc_\la)\sm\Pc_\la$,
where $|\la|\le 1$. Then $\Wc$ is said to be of
\emph{Siegel capture} type if any $f\in \Wc$ has an invariant Siegel
domain $U$ around $0$ and another Fatou domain $V$ such that $f|_V$
is two-to-one and $f^{\circ k}(V)=U$ for some $k>0$.
Also, $\Wc$ is
said to be of \emph{queer} type if the set $J(f)$ contains the
critical point $\om_2(f)$.
In this case, it can be shown (Theorem \ref{t:nosiegel}) that $J(f)$ has positive Lebesgue measure, and carries an invariant line field.
\end{dfn}

Observe that polynomials from components of Siegel capture type and
from components of queer type have connected Julia sets.

\begin{mthm} Let $\Wc$ be a bounded component of
$\thu(\Pc_\la)\sm \Pc_\la$, where $|\la|\le 1$.  Then
any polynomial $f\in \Wc$ is stable and has neither repelling
periodic cutpoints nor neutral periodic points distinct from $0$ in
$K(f)$. Moreover, $\Wc$ is either of \emph{Siegel capture} type or
of \emph{queer} type.
\end{mthm}

\medskip

{\footnotesize \emph{Notation and Preliminaries:} we write $\ol A$
for the closure of a subset $A$ of a topological space and $\bd(A)$
for the boundary of $A$; the $n$-th iterate of a map $f$ is denoted
by $f^{\circ n}$. We let $\C$ stand for the complex plane, $\C^*$
for the Riemann sphere, $\disk$ for the open unit disk in $\C$ centered at
$0$, and
$\uc=\bd(\disk)$ for the unit circle. We identify the unit circle
$\uc$ with $\R/\Z$ and denote by $\ol{\be\ga}$ the chord with
endpoints $\be$, $\ga\in \uc$. The $d$-tupling map of the unit
circle is denoted by $\si_d$. We will talk about principal sets of
arbitrary continuous paths $\ga:(0,\infty)\to\C$ such that
$\lim_{t\to\infty}\gamma(t)=\infty$, not necessarily external rays.
The principal set of $\ga$ is defined as
$\bigcap_{\eps>0}\ol{\ga(0,\eps)}$. We also assume knowledge of
basic notions from complex dynamics, such as \emph{Green function,
dynamic rays $($of specific argument$)$, B\"ottcher coordinate,
Fatou domain, repelling, attracting, neutral periodic points,
parabolic, Siegel, Cremer periodic points, impressions, principal
sets}, and the like (see, e.g., \cite{Mc}). }

\section{Critical points of potentially renormalizable polynomials}\label{s:notinphd}

%We begin with the following lemma.

%\begin{lem}\label{l:phd-conti}
%For any $\la, |\la|\le 1$, the set $\ol\phd_3\cap \Fc_\la$ is a
%continuum.
%\end{lem}

%\begin{proof}

%\end{proof}

Throughout Section~\ref{s:notinphd}, we consider a potentially
renormalizable polynomial $f=f_{\la, b}$ with $|\la|\le 1$. Recall
that, if $g\in \Ac$ is close to $f$, then $f|_{A(g)}$ is two-to-one
and contains exactly one critical point of $g$ denoted by
$\om_1(g)$. Let $\om_2(g)$ be the other critical point of $g$. Thus,
maps $g\in\Ac$ close to $f$ have two distinct critical points
with very different properties. Consistently approximating $f$ by
polynomials $g\in \Ac$, we can distinguish between critical points
of $f$ as well.
%Fix $f\in
%\Fc_{nr}$ such that $[f]\notin \ol\phd_3$.

\begin{lem}[\cite{bopt14b}]
  \label{l:2crpts}
The polynomial $f$ has two distinct critical points.
\end{lem}

\begin{proof}
Assume that $\omega(f)$ is the only critical point of $f$
(then it has multiplicity two). Let $\Cc$ be the space of all
polynomials $g\in \Fc$ with a multiple critical point $\omega(g)$.
This is an algebraic curve in $\Fc$ passing through $f$. The map
taking $g\in\Cc$ to $g'(0)$ is a non-constant holomorphic function.
Hence there are polynomials $g\in\Cc$ arbitrarily close to $f$, for
which $|g'(0)|<1$. The class of any such polynomial $g$ belongs to
$\phd_3$ as the immediate basin of $0$ with respect to $g$ must
contain the multiple critical point $\omega(g)$, contradicting
our assumption on $f$.
\end{proof}

We are ready to give the following definition.

\begin{dfn}\label{d:principo}
A critical point $c$ of $f$ is said to be \emph{principal} if there
is a neighborhood $\Uc$ of $f$ in $\Fc$ and a holomorphic function
$\omega_1:\Uc\to\C$ defined on this neighborhood such that
$c=\omega_1(f)$, and, for every $g\in\Uc\cap\Ac$, the point
$\omega_1(g)$ is the critical point of $g$ contained in $A(g)$.
\end{dfn}

Now let us prove Theorem~\ref{t:princ}.

\begin{thm}[\cite{bopt14b}]
  \label{t:princ}
There exists a unique principal critical point of $f$.
\end{thm}

\begin{proof}
By Lemma \ref{l:2crpts}, the two critical points of $f$ are
different. Then there are two holomorphic functions, $\omega_1$ and
$\omega_2$, defined on a convex neighborhood $\Uc$ of $f$ in $\Fc$,
such that $\omega_1(g)$ and $\omega_2(g)$ are the critical points of
$g$ for all $g\in\Uc$. Suppose that neither $\omega_1(f)$, nor
$\omega_2(f)$ is principal. Then, arbitrarily close to $f$, there
are cubic polynomials $g_1$ and $g_2\in\Ac$ with
$\omega_2(g_1)\not\in A(g_1)$ and $\omega_1(g_2)\not\in A(g_2)$.
Since $A(g_i)$ contains a critical point for $i=1,2$, we must have
that $\omega_i(g_i)\in A(g_i)$.

The set $\Ac$ is convex. Therefore, the intersection $\Uc\cap\Ac$ is
also convex, hence connected. Let $\Oc_i$, $i=1$, $2$, be the subset
of $\Uc\cap\Ac$ consisting of all polynomials $g$ with
$\omega_i(g)\in A(g)$. By the preceding paragraph, $g_1\in \Oc_1$
and $g_2\in \Oc_2$. We claim that $\Oc_i$ is open. Indeed, if
$g\in\Oc_i$, then there exists a Jordan disk $U\subset A(g)$ with
$g(U)$ compactly contained in $U$, and $\omega_i(g)\in U$. If
$\tilde g\in\Uc\cap\Ac$ is sufficiently close to $g$, then $\tilde
g(U)$ is still compactly contained in $U$, and $\omega_i(\tilde g)$
is still in $U$, by continuity. It follows that $U\subset A(\tilde
g)$, in particular, $\omega_i(\tilde g)\in A(\tilde g)$. Thus,
$\Oc_i$ is open. Since $\Oc_1, \Oc_2$ are open and non-empty, the
set $\Uc\cap\Ac$ is connected, and
$$
\Uc\cap\Ac=\Oc_1\cup\Oc_2,
$$
the intersection $\Oc_1\cap\Oc_2$ is nonempty. Note that
$\Oc_1\cap\Oc_2$ consists of polynomials, whose classes are in
$\phd_3$. Since $\Uc$ can be chosen arbitrarily small, it follows
that $f$ can be approximated by maps $g\in\Ac$ with $[g]\in\phd_3$, a contradiction.

The existence of a principal critical point of $f$ is thus proved.
The uniqueness follows immediately from our assumption on $f$.
\end{proof}

Denote by $\omega_1(f)$ the principal critical point of $f$.
Obviously, $\om_1(f)\in K(f)$. Let $\om_2(f)$ be the other critical
point of $f$. For $g_{\lambda,b}(z)=\lambda z+b z^2+z^3$ with
$|\la|\le 1$ (i.e., for $g\in \Fc_{nr}$) sufficiently close to $f$,
the point $\omega_1(g)$ is a holomorphic function of $g$.

\section{Parabolic dynamics} \label{s:parab}

Suppose that $0$ is a parabolic point of $f$ with rotation number
$p/q$, i.e., we have $f'(0)=\exp(2\pi i p/q)$. It follows that
$f^{\circ q}(z)=z+az^{m+1}+O(z^{m+2})$ for small $z$ and some
non-zero coefficient $a$, where $m=q$ or $2q$. A \emph{repelling
vector} is defined as a vector $v\in\C$ such that $av^m$ is a positive
real number. Repelling vectors define $m$ straight rays originating
at $0$. These rays divide the plane into $m$ open \emph{attracting
sectors}. Let $S$ be an attracting sector, and $D$ a small round
disk centered at $0$. The map $z\mapsto z^{-m}$ is defined on $S\cap
D$ and takes $S\cap D$ into a subset of the plane containing the
half-plane $\Re(w\ol a)<-M$ for some big $M>0$. Moreover, the map
$z\mapsto z^{-m}$ conjugates $f^{\circ q}|_{S\cap D}$ with a map $F$
asymptotic to $w-ma$ as $w\to\infty$. If $M$ is big enough, then $F$
takes the half-plane $\Re(w\ol a)<-M$ into itself. An
\emph{attracting petal} $P$ of $f$ at $0$ is defined as the closure
of the pullback of this half-plane to $S$. An attracting petal
depends on the choice of an attracting sector $S$ and on the choice
of the number $M$. The following properties of attracting petals are
immediate:
\begin{enumerate}
\item any attracting petal $P$ is a compact subset of the plane such that
$tP\subset P$ for $t\in [0,1]$;
\item if $P$ is an attracting petal, then the map $f^{\circ q}:P\to\C$ is univalent,
and we have $f^{\circ q}(P)\subset P$;
\item the set $f(P)$ lies in some attracting petal of $f$;
\end{enumerate}

\begin{figure}
  \includegraphics[width=6cm]{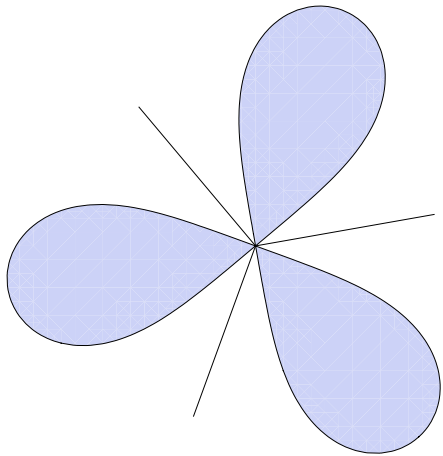}
  \label{pic:pet}
  \caption{Attracting petals and repelling directions for the
  polynomial $f(z)=e^{2\pi i/3}z+z^2+z^3$.}
\end{figure}

In what follows, given $f\in \Fc$ and small $\eps>0$, we define
$g_{f\hspace{-1pt},\,\eps}\in\Fc$ as the cubic polynomial affinely
conjugate to $(1-\eps)f$.

\begin{lem}
 \label{l:pet}
 Let $P$ be an attracting petal of $f$. If $\eps>0$ is sufficiently small,
 then $P$ is contained in $A(g_{f\hspace{-1pt},\,\eps})$.
\end{lem}

\begin{proof}
Set $g=g_{f\hspace{-1pt},\,\eps}$. Let us show that \emph{every}
attracting petal $\widetilde P$ of $f$ is contained in
$A((1-\delta)f)$ for every $0<\delta<1$. Assume that there are
attracting petals $\widetilde P_0=\widetilde P$, $\widetilde P_1$,
$\dots$, $\widetilde P_{q-1}$ with $f(\widetilde P_i)\subset
\widetilde P_{i+p\pmod q}$. It follows from property (1) of
attracting petals that the map $(1-\delta)f$ takes $\widetilde P_i$
to a subset of $\widetilde P_{i+q\pmod q}$. Hence, $\widetilde
P\subset A((1-\delta)f)$.

A conjugacy between $(1-\eps)f$ and $g$ is given by the map
$z\mapsto (1-\eps)^{1/2}z$, and we may choose $\eps$ so small that
the set $(1-\eps)^{-1/2}P$ is contained in some (slightly bigger)
attracting petal $\widetilde P$ of $f$. By the previous paragraph,
$(1-\eps)^{-1/2}P\subset A((1-\eps)f)$, hence $P\subset A(g)$.
\end{proof}

By a \emph{parabolic domain at $0$}, we mean a Fatou component of
$f$ containing some attracting petal.

\begin{cor}
  \label{c:pardom} Suppose that $f\in \Fc_\la$, $|\la|\le 1$, is
  potentially renormalizable, and $\Omega$ is a parabolic domain of
  $f$ at $0$. For every compact set $D\subset\Omega$ and every
  sufficiently small $\eps>0$, we have $D\subset
  A(g_{f\hspace{-1pt},\,\eps})$.
\end{cor}

\begin{proof}
Let $p/q$ be the rotation number, and let $P\subset \Omega$ be an
attracting petal. We may assume that $D$ is a Jordan disk. Replacing
$D$ with a bigger Jordan disk if necessary, we may assume that $D\cap
P\ne\0$. By compactness of $D$, there exists a positive integer $m$
with the property $f^{\circ qm}(D)\subset P$.

Let $\eps>0$ be a small real number, and set
$g=g_{f\hspace{-1pt},\,\eps}\in\Fc$. By Lemma \ref{l:pet}, we have
$P\subset A(g)$. We have $g^{\circ qm}(D)\subset P$ provided that
$\eps$ is small enough. It follows
that $D$ is contained in some pullback of $A(g)$. Since $D\cap P\ne\0$,
this pullback must coincide with $A(g)$.
\end{proof}

Corollary~\ref{c:1crit} identifies $\om_1(f)$ in the attracting and
parabolic cases.
%NEED TO DEFINE $Z(f)$ AND $X(f)$!!!

\begin{cor}\label{c:1crit}
Let $f\in \Fc_{nr}$ be potentially renormalizable. If $0$ is a parabolic
$($resp., attracting$)$ fixed point of $f$, and $c$ is a critical
point of $f$ belonging to a parabolic $($resp., the attracting$)$
domain $\Omega$ of $f$ at $0$, then $c=\om_1(f)$. Thus, such
$\Omega$ is unique.
%Moreover, $\Omega\subset \thu(X(f))$.
\end{cor}

\begin{proof}
We may assume that $0$ is parabolic. Let $D$ be a small disk around
$c$ contained in $\Omega$. By Corollary \ref{c:pardom}, if $\eps>0$
is small enough and $g=g_{f\hspace{-1pt},\,\eps}$, then $D\subset
A(g)$. Let $c_g$ be the critical point of $g$ close to $c$. If
$\eps$ is small enough, then $c_g\in D$. Therefore, $c_g\in A(g)$
and $c=\om_1(f)$ by definition of the principal critical point.
\end{proof}

The proof of Corollary~\ref{c:near} is left to the reader; notice that
the claim about the Julia set of a polynomial $g$ being locally
connected follows from \cite{DH}.

\begin{cor}\label{c:near}
Suppose that $f\in \Fc_{nr}$ is potentially renormalizable, $0$ is parabolic and,
for some $k$, the point $f^{\circ k}(\om_2(f))$ belongs to a parabolic
domain at $0$. Then the maps $g=g_{f\hspace{-1pt},\,\eps}$ converge
to $f$ as $\eps\to 0$, have locally connected Julia sets and are
such that $\om_1(g)\in A(g)$, $\om_2(g)\notin A(g)$ and
$g^{\circ k}(\om_2(g))\in A(g)$.
\end{cor}

%\section{Bounded potentially renormalizable components must be
%of Siegel capture type or of queer type}\label{s:exteclosu}

\section{Bounded components of $\Fc_\la\sm \Pc_\la$ must be
of Siegel capture type or of queer type}\label{s:exteclosu}

We need the notion of an active critical point introduced by
McMullen in \cite{Mc00}. Set $i=1$ or 2, and take $f\in\Fc_\lambda$.
The critical point $\om_i(f)$ is \emph{active} if, for every
neighborhood $\Uc$ of $f$ in $\Fc_\lambda$, the sequence of the
mappings $g\mapsto g^{\circ n}(\om_i(g))$ fails to be normal in
$\Uc$. If the critical point $\om_i(f)$ is not active, then it is
said to be \emph{passive}. %We will use this notion in the proof of
%the following corollary.

\begin{cor}[\cite{bopt14b}]
\label{c:bdd-stab}
  Let $\lambda$ be a non-repelling multiplier.
  %Every bounded potentially renormalizable component $\Wc$ in $\Fc_\lambda$
  Every bounded component $\Wc$ of $\Fc_\la\sm \Pc_\la$
  consists of stable maps.
\end{cor}

\begin{proof}
  By \cite{MSS}, to prove that $f\in\Wc$ is stable, it suffices to show that
  both critical points of $f$ are passive.
  Note that, if $g\in\bd(\Wc)$, then the $g$-orbits of $\om_1(g)$ and
  of $\om_2(g)$ are bounded uniformly with respect to $g$.
  By the maximum principle, the $f$-orbits of $\om_1(f)$ and $\om_2(f)$
  are uniformly bounded for all $f\in\Wc$, which implies normality.
  Thus both critical points are passive, and the corollary is proved.
\end{proof}

We want to improve the description of the dynamics of maps in a
bounded %potentially renormalizable
component of $\Fc_\la\sm \Pc_\la$ given in \cite{bopt14}
(see Theorem~\ref{t:extendclo}). Let $f\in\Fc_\lambda, |\lambda|\le
1$, belong to a bounded %potentially renormalizable
component
$\Wc_f$ of $\Fc_\la\sm \Pc_\la$. By Corollary~\ref{c:bdd-stab}, the map $f$ is stable. A
priori, $\Wc_f$ (and $f$) can be classified into five types:
\begin{description}
  \item[Disjoint type] the critical point
$\omega_2(f)$ lies in a periodic attracting basin but not in $A(f)$.
\item[Attracting capture type] we have
$|\lambda|<1$, and $\omega_2(f)$ is \emph{eventually} mapped to
$A(f)$.
\item[Parabolic capture type]
$\lambda$ is a root of unity, and $\omega_2(f)$ is eventually mapped
to a periodic parabolic Fatou domain at $0$.
\item[Siegel capture type]
the critical point $\omega_2(f)$ is eventually mapped to the Siegel
disk around $0$.
\item[Queer type]
we have $\omega_2(f)\in J(f)$.
\end{description}

In this section, we prove that %potentially renormalizable
bounded components of $\Fc_\la\sm \Pc_\la$
of the first three types do not exist. We use \cite{bfmot12}, where
fixed and periodic points of various maps of plane continua were
studied.

\begin{dfn}\label{d:weakrep}
A \emph{dendrite} is a locally connected continuum containing no
Jordan curves. If $g:D\to D$ is a self-mapping of a dendrite $D$, a
periodic point $a$ is said to be \emph{weakly repelling} is there
exists an arc $I\subset D$ with endpoint $a$ and a number $k$ such
that $g^{\circ k}(a)=a$ and, for any $x\in I\sm \{a\}$, the point
$x$ separates the points $a$ and $g^{\circ k}(x)$ in $D$.
\end{dfn}

%Having all periodic points weakly repelling is a strong condition.

\begin{thm}[Theorem 7.2.6 \cite{bfmot12}]\label{t:726}
Suppose that $g:D\to D$ is a self-mapping of a dendrite $D$ such
that all its periodic points are weakly repelling. Then $g$ has
infinitely many periodic cutpoints. \end{thm}

Lemma~\ref{l:repexist} has a topological nature.

\begin{lem}\label{l:repexist}
If a polynomial $F$ has locally connected Julia set with pairwise
disjoint closures of its Fatou domains, then $J(F)$ contains
infinitely many periodic repelling cutpoints unless $J(F)$ is a
Jordan curve.
\end{lem}

\begin{proof}
Assume that $J(F)$ is not a Jordan curve. Define a map $\psi$ which
collapses Fatou domains of $F$ to points and creates a dendrite $D$
out of $J(F)$ with the map $h:D\to D$ induced by $F$; thus,
$\psi\circ F=h\circ \psi$ and $F$ is semiconjugate to $h$ (if $J(F)$
is already a dendrite, then $\psi$ can be taken as the identity map). Then all
periodic points of $h$ are weakly repelling. Indeed, the
corresponding property is immediate if an $h$-periodic point is the
$\psi$-image of an $F$-periodic point. Otherwise let $U$ be a
periodic Fatou domain of $F$. Observe that, by the assumption, $U$
cannot have a critical point, say, $c$ in $\bd(U)$ since if it does,
then $F$ will have two Fatou domains, whose closures contain $c$, a
contradiction.

We claim that there exists a periodic cutpoint $x$ of $J(F)$ with
$x\in \bd(U)$. Indeed, we may assume that $U$ is $F$-invariant. Take
a component $K$ of $J(F)\sm \ol{U}$ (it exists since $J(F)\ne
\bd(U)$). Then $\ol{K}\cap \ol{U}=\{y\}$ for some point $y$. If $y$
is not (pre)periodic, then there exists $N$ such that for any $n\ge
N$ any component of $J(F)\sm \{F^{\circ n}(y)\}$ except for one component
containing $U$ contains no critical points. This implies that any
such component $T$ is wandering. However, $T$ must contain
pullbacks of $U$, a contradiction. Hence $y$ is
(pre)periodic, and some image $x$ of $y$ is a desired periodic
cutpoint of $J(F)$ of period, say, $m$. Since $U$ is invariant, the
combinatorial rotation number of $F^{\circ m}$ at $x$ is zero. This and the
fact that $x$ is repelling in $J(F)$ implies that the point
$\psi(\ol{U})$ is a weakly repelling periodic point of $h$.

Thus, any periodic point of $h$ is weakly repelling. By
Theorem~\ref{t:726}, the map $h$ has infinitely many periodic cutpoints in
$\psi(J(F))$. By construction, this implies that $F$ has infinitely
many periodic repelling cutpoints (recall that $F$ has only finitely
many periodic Fatou domains).
\end{proof}

Corollary~\ref{c:infperio} now easily follows.

\begin{cor}\label{c:infperio}
Let $g\in \Ac$ be such that $g|_{A(g)}$ is two-to-one and
$g^{\circ k}(\om_2(g))\in A(g)$. Then $g$ has locally connected Julia set
with pairwise disjoint closures of its Fatou domains. Thus, $J(g)$
contains infinitely many periodic repelling cutpoints.
\end{cor}

\begin{proof}
By \cite{DH}, the set $J(g)$ is locally connected. Suppose that two
Fatou domains $U$, $V$ are such that $\ol U\cap \ol V\ne \0$. Since
$A(g)$ is the only periodic Fatou domain of $g$, both $U$ and $V$
must eventually map to $A(g)$. Hence there exists a unique number
$m$ such that $g^{\circ m}(U)\ne g^{\circ m}(V)$ while
$g^{\circ m+1}(U)=g^{\circ m+1}(V)$. This
implies that the common point of $g^{\circ m}(\ol U)$ and $g^{\circ m}(\ol V)$ is
critical, a contradiction. Now Lemma~\ref{l:repexist} implies the
remaining claims of this lemma.
\end{proof}

To prove that %potentially renormalizable
bounded components of $\Fc_\la\sm \Pc_\la$ cannot be of
disjoint type or attracting (parabolic) capture types, we use
perturbations as a tool. Thus, we need a few general facts about
perturbations. Lemma~\ref{l:rep} goes back to Douady and Hubbard
\cite{DH}.
Recall that a \emph{smooth $($external\,$)$ ray} of a polynomial $P$, for which $J(P)$
is not necessarily connected, is defined as a homeomorphic image of $\R$
in $\C\sm K(P)$ accumulating only in $\{\infty\}\cup J(P)$, avoiding pre-critical
points of $P$, and tangent to the gradient of the Green function of $P$.
An external ray of argument $\theta\in\R/\Z$ is denoted by $R_P(\theta)$.
%Note that if an external ray lands, then it is smooth.

\begin{lem}[Lemma B.1 \cite{GM}] \label{l:rep}
Let $f$ be a polynomial, and $z$ be a repelling periodic point of
$f$. If a smooth \emph{periodic} ray $R_f(\theta)$ lands at $z$,
then, for every polynomial $g$ sufficiently close to $f$, the ray
$R_{g}(\theta)$ is also smooth, lands at a repelling periodic point
$w$ close to $z$, and $w$ depends holomorphically on $g$.
\end{lem}

Lemma~\ref{l:converge} now easily follows.

\begin{lem}\label{l:converge}
Suppose that $h_n\to h$ is an infinite sequence of polynomials of
degree $d$, and $\{\al, \be\}$ is a pair of periodic arguments such
that the external rays $R_{h_n}(\al)$, $R_{h_n}(\be)$ land at the
same repelling periodic point $x_n$ of $h_n$. If $R_h(\al)$,
$R_h(\be)$ are smooth but do not land at the same periodic point of $h$, then one
of these two rays must land at a parabolic point of $h$.
\end{lem}

\begin{proof}
We may assume that $x_n$ converge to an $h$-periodic point $x$. If
both rays $R_h(\al)$, $R_h(\be)$ land at distinct \emph{repelling}
periodic points, then, by Lemma~\ref{l:rep}, we get a contradiction
with the fact that $R_{h_n}(\al)$, $R_{h_n}(\be)$ land at $x_n$ and
$x_n\to x$. Since $\al$ and $\be$ are periodic, by the Snail Lemma, $R_h(\al)$ and
$R_h(\be)$ cannot land at a Cremer point. Hence one of the rays
$R_h(\al)$, $R_h(\be)$ must land at a parabolic periodic point.
\end{proof}

The next lemma deals with much more specific perturbations.

\begin{lem}\label{l:pertb0}
Let $f\in \Fc_{nr}$ be potentially renormalizable such that the
point $0$ is parabolic. Suppose that $f^{\circ k}(\om_2(f))$ belongs
to a parabolic Fatou domain at $0$ for some $k>0$. Then there are
infinitely many repelling periodic cutpoints in $J(f)$.
\end{lem}

%By Lemma~\ref{l:pet} $\om_1(f)$ belongs to a periodic Fatou domain
%at $0$ while $\om_2(f)$ does not belong to any periodic Fatou domain
%at $0$.

\begin{proof}
By Corollary~\ref{c:near}, we can choose a sequence $\eps_n\to 0$ so that
the maps $g_n=g_{f\hspace{-1pt},\,\eps_n}$ have the
following properties:
\begin{enumerate}
\item each $g_n$ has an attracting fixed point $0$ and a critical point
$d_n\approx \om_2(f)$ such that $d_n\notin A(g_n)$ and $g_n^{\circ k}(d_n)\in A(g_n)$;
\item the maps $g_n$ are pairwise
topologically conjugate on their Julia sets.
\end{enumerate}
The second requirement can be fulfilled because for any sequence
$g_n$ satisfying $(1)$, the topological dynamics of $g_n:J(g_n)\to
J(g_n)$ can be of finitely many types (observe that there are
finitely many Thurston invariant laminations modeling polynomials
with property $(1)$).

Now, by Corollary~\ref{c:infperio}, there are infinitely many pairs of
periodic angles $\al_i$, $\be_i$ such that the external rays
$R_{g_n}(\al_i), R_{g_n}(\be_i)$ land at the same periodic repelling
cutpoint $x_i(g_n)$ of $J(g_n)$. By Lemma~\ref{l:converge}, this
implies that, for infinitely many subscripts $i$, the points
$x_i(g_n)$ converge to a repelling periodic cutpoint $x_i(f)$ of
$J(f)$, and all these points are distinct.
\end{proof}

We will need the following lemma.

\begin{lem}\label{l:noneut}
If $\Vc$ is a stable domain in $\Fc_\la$, then $f\in \Vc$ does not
have a periodic neutral point other than $0$.
\end{lem}

\begin{proof}
If $f$ is stable with a neutral periodic point $x=x(f)\ne 0$, then it follows that
all maps $g\in \Vc$ have a neutral periodic point $x(g)\ne 0$ of the same period. Clearly,
this implies that the multiplier of $x(g)$ for $g\in \Vc$ must be constant, a contradiction.
\end{proof}

We are ready to prove Lemma~\ref{l:no3types}.

\begin{lem}\label{l:no3types}
Let $\Wc$ be a component of $\thu(\Pc_\la)\sm \Pc_\la$. Then any
polynomial $f\in \Wc$ has connected Julia set $J(f)$ and has neither
repelling periodic cutpoints nor non-repelling periodic points
distinct from $0$. In particular, $\Wc$ is either of Siegel capture
type or of queer type.
%such that $\om_2(f)\in J(f)$.
\end{lem}

\begin{proof}
Assume that $f\in \Wc$. Since the critical points of polynomials from $\bd(\Wc)$
are non-escaping, it follows that the critical points of $f$ are also non-escaping.
Hence $J(f)$ is connected. Consider now two possibilities. Recall that, by
Corollary~\ref{c:bdd-stab}, all maps in $\Wc$ are stable.

First, it may happen that $\Wc$ is a stable component. Then by
Definition~\ref{d:exte} where the extended closure $\ol{\phd}_3^e$ of
the cubic Principal Hyperbolic Domain $\ol{\phd}_3$ is defined it follows
that the polynomial $f$ belongs to $\ol{\phd}_3^e$.
Hence, by Theorem~\ref{t:prophd}, the polynomial $f$ cannot have
attracting periodic points distinct from $0$ or repelling periodic
cutpoints distinct from $0$ of its Julia set. On the other hand, by
Lemma~\ref{l:noneut}, the polynomial $f$ cannot have neutral periodic points distinct
from $0$. This completes the proof of the lemma in the case when $\Wc$
is a stable component.

Second, suppose that $\Wc$ is not a stable component. Then there is
a unique stable component $\Vc$ in $\Fc_\la$ such that
$\Vc\supsetneqq \Wc$. It follows that there is a polynomial $g\in
\bd(\Wc)\cap \Vc$; in other words, there is a polynomial $g\in\Vc$ such that
$[g]\in\ol\phd_3$, and the maps $f|_{J(f)}$ and $g|_{J(g)}$ are
quasi-symmetrically conjugate.

Suppose that $f$ has a repelling periodic cutpoint $x$ of its Julia
set. Clearly, $x\ne 0$. The corresponding periodic point $y$ of $g$
cannot be repelling by Theorem~\ref{t:prophd}. Since $f|_{J(f)}$ and
$g|_{J(g)}$ are conjugate, it follows that the only possibility for
$y\ne 0$ is that $y$ is a parabolic periodic point. However this
contradicts Lemma~\ref{l:noneut}. Hence $f$ has no repelling
periodic cutpoints of its Julia set. Moreover, by
Lemma~\ref{l:noneut}, the polynomial $f$ has no neutral periodic points distinct
from $0$.

Finally, suppose that $f$ has an attracting periodic point distinct
from $0$. Then the fact that $f$ and $g$ are quasi-conformally
conjugate on their Julia sets implies that $g$ has either (1) an
attracting periodic Fatou domain $U$ not containing $0$ or (2) a
parabolic periodic Fatou domain with a parabolic point $z\ne 0$ on
its boundary. Since $[g]\in \ol{\phd_3}$, case (1) is impossible by
Theorem~\ref{t:prophd}. On the other hand, case (2) is impossible by
Lemma~\ref{l:noneut}. This completes the proof of the first claims
of the lemma.

Let us use this to show that $\Wc$ is either of Siegel capture type
or of queer type. Indeed, if $f\in \Wc$ is of disjoint type, then $f$
has an attracting periodic point distinct from $0$, which is
impossible by the above. If $f$ is of attracting capture type, then,
by Corollary~\ref{c:infperio}, the set $J(f)$ contains infinitely
many periodic repelling cutpoints, again a contradiction with the
above. Finally, if $f$ is of parabolic capture type, then, by
Lemma~\ref{l:pertb0}, we see that $f$ has infinitely many periodic
repelling cutpoints, and we obtain the same contradiction with the
above. Therefore, $\Wc$ is either of Siegel capture type or of queer
type.
\end{proof}

%Now we prove Theorem~\ref{t:nosiegel}.
The following theorem describes components of queer type.

\begin{thm}
  \label{t:nosiegel}
Let $\Wc$ be a bounded stable component of $\Fc_\la$ of queer type.
Then, for any polynomial $f\in\Wc$, the Julia set $J(f)$ has positive Lebesgue measure and carries an invariant line field.
\end{thm}

The most difficult case of Theorem \ref{t:nosiegel} is covered by the following theorem of S. Zakeri:

\begin{thm}[\cite{Z}, Theorem 3.4]
\label{t:Zak}
Let $0<\theta<0$ be a Brjuno number, and set $\la=e^{2\pi i\theta}$.
Consider a polynomial $f$ from a queer component of $\Cc_\la\sm\bd(\Cc_\la)$.
Then  $J(f)$ pas positive Lebesgue measure and carries an invariant line field.
\end{thm}

Brjuno numbers are irrational numbers satisfying a certain number theoretic condition.
If $\theta$ is Brjuno, then any holomorphic germ at $0$ of the form $z\mapsto \la z+\dots$, where dots denote terms of degree 2 and higher, is linearizable \cite{Yoc}.
The converse is known only for quadratic polynomials.
The proof of \cite[Theorem 3.4]{Z} does not use the Brjuno condition.
It works verbatim for any stable component $\Wc$ of $\Fc_\la$ such that polynomials
$f\in\Wc$ have a Siegel disk around $0$ (obviously, this condition does not depend
on the choice of $f$ within $\Wc$).

\begin{comment}
In fact, the proof of \cite[Theorem 3.4]{Z} extends almost verbatim
to a proof of Theorem \ref{t:nosiegel}. One needs only to replace
the grand orbit of the Siegel disk with the union of all bounded
Fatou components. The proof is based on the fact that the
holomorphic motion of the Julia set admits a canonical equivariant
extension to the entire Riemann sphere. This is easy if there are no
Siegel disks, and most difficult if there is a Siegel disk, whose
boundary is not locally connected. Thus \cite[Theorem 3.4]{Z} works
out only the difficult case. Note that it suffices to consider a
Siegel disk around $0$, since the existence of other Siegel disks
contradicts Lemma~\ref{l:no3types}.
\end{comment}

\begin{proof}[Proof of Theorem \ref{t:nosiegel}]
The argument follows the same lines as \cite{McS,Z}.
There is an equivariant holomorphic motion $\mu$ of the set $J(f)$ for $f\in\Wc$.
If we fix some $f_0\in\Wc$, then $\mu$ can be regarded as a map from $J(f_0)\times\Wc$
to $\C$ such that $\mu(z,f_0)=z$ for all $z\in J(f_0)$ and, for any fixed $z$, the
map $f\mapsto \mu(z,f)$ is holomorphic.
Suppose that $\mu$ can be extended to a holomorphic motion of the entire complex plane
in such a way that $z\mapsto \mu(z,f)$ is holomorphic on the Fatou set of $f_0$, for
any $f$ sufficiently close to $f_0$.
Then $J(f_0)$ must have positive Lebesgue measure (otherwise, by the Ahlfors Lemma,
the map $z\mapsto \mu(z,f)$ would be a global holomorphic conjugacy between $f_0$ and $f$).
An invariant line field on $J(f_0)$ is given by the pullback of the standard conformal
structure under the quasi-conformal map $z\mapsto \mu(z,f)$: say, one can take major axes
of the ellipses representing the described quasi-conformal structure.

The desired extension of $\mu$ is easy if there are no bounded Fatou components of $f_0$.
Indeed, in this case, it suffices to define $\mu(z,f)$ for $z$ outside of $K(f_0)$.
We set $\mu(z,f)$ to be a point in the dynamical plane of $f$, whose (suitably normalized) B\"ottcher coordinate with respect to $f$ is the same as the B\"ottcher coordinate of $z$ with respect to $f_0$.
If the Fatou set of $f_0$ has bounded components, then we can also extend $\mu$ to each
of the Fatou components by equating the (suitably normalized) linearizing coordinates.
The only problem is with Siegel disks, for which the linearizing coordinates have no obvious normalization, especially if the boundary of the Siegel disk is not locally connected and there are no critical points on the boundary.
This (most difficult) case is covered by \cite[Theorem 3.4]{Z} stated above as Theorem \ref{t:Zak}.
\end{proof}

Obviously, Lemma~\ref{l:no3types} and Theorem~\ref{t:nosiegel} imply
our Main Theorem.

\section*{Acknowledgements}
The authors are grateful to M. Lyubich for useful discussions. We
are also indebted to the referee for thoughtful remarks.

The first and the third named authors were partially supported by NSF
grant DMS--1201450. The second named author was partially  supported by
NSF grant DMS-0906316. The fourth named author was partially supported
by the RFBR grant 16-01-00748a.
The article was prepared within the framework of a subsidy granted to the
HSE by the Government of the Russian Federation for the implementation
of the Global Competitiveness Program.

\end{document}